\documentclass[12pt]{article}

\usepackage{setspace}
\usepackage{textcomp}
\usepackage[dvips]{color}
\usepackage{epsfig}
\usepackage{amsmath, amsthm, amssymb}
\usepackage[left=1.5in,top=.9in, 
right=1in,nohead]{geometry}
\usepackage{rawfonts}
\usepackage[all]{xy}
\usepackage[sc,osf]{mathpazo}  
\usepackage{multirow}
\usepackage{pgf,tikz}
\usetikzlibrary{arrows,cd,matrix,decorations.pathmorphing}
\usetikzlibrary{babel}            
\usepackage{hyperref}

\usepackage{comment}

\def\bct{\begin{center}}
\def\ect{\end{center}}
\def\beg{\begin}

\def\<{\langle}
\def\>{\rangle}

\def\mbb{\mathbb}

\def\mbbr{\mathbb R}
\def\mbbz{\mathbb Z}

\def\mf{\mathfrak}
\def\ni{\noindent}

\def\tn{\textnormal}

\newtheorem{thm}{Theorem}[section]
\newtheorem{lem}[thm]{Lemma}
\newtheorem{prop}[thm]{Proposition}
\newtheorem{cor}[thm]{Corollary}
\newtheorem{rmk}[thm]{Remark}
\newtheorem{defn}[thm]{Definition}

\def\R{{\mathbb R}}

\title{Algebraic Topology of Special Lagrangian Manifolds}
\author{Mustafa Kalafat \and Ey\"{u}p Yal\c{c}{\i}nkaya}

\begin{document}

\maketitle

\begin{abstract} In this paper,  we prove various results 
on the topology of the Grassmannian 
of oriented 3-planes in Euclidean 6-space 
and 
compute its cohomology ring. 
We give self-contained proofs. These spaces come up when studying submanifolds of manifolds with calibrated geometries. We collect these results here for the sake of completeness. 
As applications of our algebraic topological study we present some results on special Lagrangian-free embeddings of surfaces and 3-manifolds into the Euclidean 4 and 6-space.

\vspace{.05in}

\ni {\em Keywords:} Calibrations; special holonomy; fiber bundles; Grassmannians.

\vspace{.05in}

\ni {\em Mathematics Subject Classification 
2010:} Primary 53C38; Secondary 57R20, 57R22.

  \end{abstract}

\section{Introduction}
This paper is devoted to the algebraic topological and geometric study of the space of oriented 3-planes in 6-dimensional real Euclidean space which we denote by $G_3^+\mbbr^6$. These manifolds are traditionally named as {\em Grassmannians}. We will give some definitions in the subject first, 
interested reader would like to consult to the fundamental article \cite{harveylawson} of Harvey and Lawson  or \cite{joyceSLclay} for more background. Let $ G_k^+\mbbr^n$ be a Grassmannian manifold defined by all oriented $k$-dimensional subspaces of $ \mbbr^n.$ For any $v \in G_k^+\mbbr^n $ there are orthonormal vectors 
$e_1\cdots e_k$ such that $v=e_1\wedge\cdots\wedge e_k.$
Let $\phi$ be a closed $k$-form on $\mbbr^n $. If $|\phi(v_1 \wedge v_2 \wedge\cdots \wedge v_k)|\leq 1$ for any orthonormal set of vectors $v_1, v_2, . . . , v_k \in  \mbbr^n$, then $\phi$ is called a {\em calibration} on $ \mbbr^n.$ The set 
$$\{ v_1\wedge v_2\wedge\cdots\wedge v_k \in G_k^+\mbbr^n \ | \ \phi (v_1\wedge v_2\wedge\cdots\wedge v_k)=\pm1\}$$ 
is called {\em the contact set} or {\em face} of the calibration $\phi$. A k-submanifold is called {\em calibrated} if its tangent subplanes are in the contact set. 
An example of a calibration on $\mbb R^{2n} \simeq \mathbb{C}^n$ 
is the real n-form $\alpha = Re\{dz_1\wedge\cdots\wedge dz_n\}.$ The calibrated submanifolds in this geometry  are Lagrangian submanifolds of $\mathbb{C}^n$ which satisfy an additional  'determinant' condition. They are therefore called {\em special Lagrangian submanifolds}. They, of course, have the property of being absolutely area-minimizing. 

\vspace{.05in}


\indent Let $(X,\varphi)$ be a calibrated manifold. A $p$-plane $\xi$ is said to be {\em tangential} to a submanifold $M
\subset X$ if $span \xi \subset T_xM$ for some $x \in M$.
A closed submanifold $M \subset X$ is called {\em $\varphi$-free} if
there are no $\varphi$-planes $\xi \in G(\varphi)$ which are tangential to $M$. 
Each submanifold of dimension strictly
less than the degree of $\varphi$ is automatically $\varphi$-free. Locally, generic $p$-dimensional submanifolds are $\varphi$-free. Depending on the calibration, there is an upper bound for the dimension of a $\varphi$-free submanifold. 
The \textit{free dimension} $fd(\varphi)$ of a calibrated manifold $(X,\varphi)$ is the maximum dimension of a linear subspaces in $TX$ which contains no $\varphi$-planes.  Subspaces which satisfy such condition are called {\em $\varphi$-free}. 
Hence, the dimension of a $\varphi$-free submanifold can not exceed {\em $fd(\varphi)$}. For all well-known calibrations on manifolds with special holonomy, this dimension is computed and shown in the Table 
\ref{table:freedimensions} of free dimensions. See \cite{HLpotentialtheory} for the details.

\beg{table}[ht] \caption{ {\em Free dimensions.
}}
\bct{\Large {\renewcommand*{\arraystretch}{1.5}
\hspace{-.2cm}\resizebox{14cm}{!}{
\begin{tabular}{|c|c|c|c|}\hline
\textbf{Calibration}&\textbf{Explanation}& \textbf{Dimension}&\textbf{Free dimension }\\ \hline
$\omega^p/p!$ & K\"ahler and related 2p-forms & 2n & $n+p-1$ \\\hline
Re($\Omega$)& $\Omega$ = Holomorphic Volume n-form &2n& $2n-2$  \\\hline
$\displaystyle\Psi_p=({\omega_I^2}+{\omega_J^2}+{\omega_K^2})^p/(2p+1)!$
& Quaternionic 
4p-calibrations & 4n &$3(n-p+1)$\\ \hline
$\phi$ &   associative 3-form &7& 4\\ \hline
$*\phi$&co-associative 4-form &7& 4\\ \hline
$\psi$ &   Cayley 4-form      &8& 4\\ \hline
\end{tabular}}} }\ect
\label{table:freedimensions} \end{table}

\ni $\varphi$-free submanifolds are the generalization of totally real submanifolds in complex geometry to calibrated manifolds.  
They are used to construct Stein-like domains in calibrated manifolds, called as $\varphi$-convex domains.

\vspace{.05in}

In our paper we investigate the invariants of this important Grassmannian. In particular we compute its integral homology as a main result of our paper. 

\vspace{.05in}

\noindent {\bf Theorem \ref{homologyG3+R6}.} 
{\em The homology of the oriented Grassmann manifold $G_{3}^{+}\mathbb{R}^{6}$ is given by  \begin{equation*}
H_*(G_3^+\mbbr^6;\mbbz)=(\mbbz,0,\mbbz_2, 0 ,
\mathbb{Z}, \mbbz ,\mbbz_2,0,0,\mathbb{Z})  \end{equation*}
}
\noindent We also compute its cohomology ring. 

\vspace{.05in}

\noindent {\bf Theorem \ref{G3R6ring}.} 
{\em The cohomology ring of the Grassmannian $G_3^+\mbbr^6$ is as follows where $\tn{deg}x_m=\tn{deg}y_m=m$.
$$\hspace{-6mm}H^*(G_3^+\mbbr^6;\mbbz)=
\mbbz[x_4,x_5]/\langle x_4^2,x_5^2, x_4x_5-x_5x_4 \rangle \oplus \mbbz_2[y_3,y_7]/\langle y_3^2,y_7^2,y_3y_7  
\rangle.$$ }
Along the way we also compute the rings of some of the Stiefel manifolds. 

\vspace{.05in}

\noindent {\bf Corollary \ref{Stiefelring}.}  
{\em The cohomology ring of the Stiefel manifolds are the following for which $\tn{deg}x_m=m$. 
\begin{enumerate}
\item $H^*(V_2\mbbr^5;\mbbz)=\mbbz[x_7]/(x_7^2)\oplus\mbbz_2[x_4]/(x_4^2)$
\item $H^*(V_3\mbbr^6;\mbbz)=\mbbz[x_5,x_7]/(x_5^2,x_7^2)\oplus\mbbz_2[x_4,x_9]/(x_4^2,x_9^2,x_4x_9)$
\end{enumerate} }

\ni In addition we compute some homotopy groups as well. 

\vspace{.05in}

\noindent {\bf Lemma \ref{lempi3}.}  
{\em The preliminary homotopy groups of the Grassmannian are the following.  $$\pi_{01234}\, G_3^+\mbbr^6=(0,0,\mbbz_2,\mbbz_2,\mbbz).$$ }
Using Serre's spectral sequence we also compute the invariant for the {\em special Lagrangian manifold SLAG}, 
which is defined to be the set of 3-planes of maximal (or minimal) energy in 6-space. It is a 5-dimensional submanifold of the Grasssmannian. 

\vspace{.05in}

\noindent {\bf Corollary \ref{SLAGring}.} 
{\em The cohomology ring of the special Lagrangian manifold is the following 
truncated polynomial ring for which $\tn{deg}x_m=m$.
\begin{eqnarray*}
H^*(SLAG;\mbbz) 
 & = & \mbbz[x_3,x_5]/(2x_3,x_3^2,x_5^2,x_3x_5)\\ [2\jot]
 & = & \mbbz[x_5]/(x_5^2)\oplus\mbbz_2[x_3]/(x_3^2). \nonumber
\end{eqnarray*}}
As an outcome of this algebraic topological study of Grassmannians we continue with the following application. 

\vspace{.05in} 

\noindent {\bf Corollary \ref{G2R4SLfree}.} 
{\em  A closed orientable surface $M$ can be embedded into $\mathbb{R}^4\cong \mathbb{C}^2$ as a sLag-free submanifold if and only if the Euler characteristic of $M$, $\chi (M) = 0$. }

\vspace{.05in} 

As another application we can talk about SL-free embeddings of 3-manifolds into Euclidean 6-space. 

\vspace{.05in}

\noindent {\bf Theorem \ref{G3R6SLfree}.} 
{\em  Let $M^3$ be a closed, oriented 3-manifold and 
$i:M\to\mbbr^6$ be an immersion, then the image $g_N(M)$ of 
the normal bundle under the normal Gauss map $g_N :M\to G_3^+\mbbr^6$ is
contractible and the normal bundle of the immersed 
submanifold is trivial, the  immersion is generically special 
Lagrangian free.}

\vspace{.05in}

There is a growing interest in Grassmannian manifolds due to their role in calibrated geometries. Interested reader may consult to 
\cite{coassf} for an example. See also \cite{cp5} for cohomology of Grassmannians.   
This paper is organized as follows. In section \S\ref{secstiefel} we deal with the related Stiefel manifolds, in section \S\ref{sechomotopy} with some homotopy theory, in section  \S\ref{secgrassmann} with the Grassmann manifold, in section \S\ref{secslag} with the special Lagrangian submanifold, in \S\ref{secring} with the cohomology ring, and finally in \S\ref{secslfreedim2}-\ref{secnormalbundle} with some geometric applications. 

\vspace{.05in}

\textbf{Acknowledgements.} We would like to thank M. Kreck for useful discussions, K. Mohnke for 
motivating us to work on this problem, in particular pointing out the paper \cite{haefliger}, and \.{I}. {\"U}nal. 
The first author would like to thank his father and family for their support during the course of this paper. This work is partially supported by the grant $\sharp$114F320 of T\"ubitak \footnote{Turkish science and research council.}.

\newpage

\section{Stiefel Manifolds}\label{secstiefel}
In order to compute the invariants of Grassmannian manifolds, some knowledge about the related Stiefel manifolds is necessary. That is why we are going to study these manifolds in this section. 
We start with a simpler one. Namely $V_2\mbbr^5$, the bundle of ordered orthonormal 2-frames in the Euclidean 7-space. 
We start with the following proposition. 

\begin{prop}\label{homologyV2R5}
The homology of the  Stiefel manifold $V_2\mbbr^5$ is the following.
$$ H_*(V_2\mathbb{R}^5;\mbbz)=(\mathbb{Z},0,0,\mathbb{Z}_2,0,0,0,\mathbb{Z}).$$
\end{prop}
\begin{proof} Using the cellular decomposition of Stiefel manifolds, a proof of this fact is presented at \cite{hatcherat}. For warming up purposes for the following cases we present a different proof here. 
The initial homotopy groups $\pi_{0123}V_2\mathbb{R}^5=(0,0,0,\mbbz_2)$ since our 7-dimensional Stiefel manifold 
is 5-2-1=2-connected and \cite{paechter1956}. Consequently, 
by the Hurewicz theorem we determine the homology groups upto the 3-rd level. 
As the next step, we are going to use the following fibration 
\begin{equation} \mbb S^3 \to V_2\mathbb{R}^5\longrightarrow\mbb S^4 
\label{V2R5spherefibration}	\end{equation}
\ni for the rest of the homology groups. The homological Serre spectral sequence of the fibration (\ref{V2R5spherefibration}) is defined together with the description of its limit as follows.
$$E^2_{p,q}:=H_p(\mbb S^4;H_q(\mbb S^3;\mbbz))$$
$$E^\infty_{p,q} = F_{p,q}/{F_{p-1,q+1}}$$ 

\ni where the abelian groups $F_{p,q}$ are defined through  
\begin{equation*}F_{p,q}:= \tn{Im}\{H_{p+q}(V^p;\mbbz)\rightarrow H_{p+q}(V;\mbbz)\}\label{filtration}\end{equation*} 
satisfies a filtration condition,
\begin{equation}0=F_{-1,n+1}\subset\cdots\subset F_{n-1,1}\subset F_{n,0}=H_n(V;\mbbz).\label{filtrationhomology}\end{equation}
Out of this information the prior pages of our spectral sequence reads as in Table \ref{table:V2R5hom}.
\beg{table}[ht] \caption{ {\em Homological Serre spectral sequence for $V_2\mbbr^5$.
}}
\bct{
$\begin{array}{cc|c|c|@{} c @{}|c|c|} \cline{3-7}
    &3& \mbbz &       &   &&\mbbz\\ \cline{3-7}
    &2&       &       &\times 2&&\\ \cline{3-7}
    &1&       &       & ~~ &~~&~~\\ \cline{3-7}
E^4~~~&0&\mbbz&       &   &  &\mbbz\\ \cline{3-7}
& \multicolumn{1}{c}{ } & \multicolumn{1}{c}{0} 
& \multicolumn{1}{c}{1} & \multicolumn{1}{c}{2} 
& \multicolumn{1}{c}{3} & \multicolumn{1}{c}{4}\\  
\end{array}$   
\setlength{\unitlength}{1mm}
\begin{picture}(0,0)
\put(-10,-2){\vector(-3,2){18}}   
\end{picture} }
\ect
\label{table:V2R5hom} \end{table}

\ni Since we know that $H_3(V;\mbbz)=\mbbz_2=F_{3,0}$ coupling with the information $E^\infty_{3,0}=0$ we receive that $F_{2,1}=\mbbz_2$.  
Continuing in this direction we finally reach at $F_{0,3}=\mbbz_2$.  
The equality revealing the limit 
$$E^\infty_{0,3}=F_{0,3}/F_{-1,4}=F_{0,3}=\mbbz_2\approx\mbbz/\tn{Im}\,d^4_{4,0}$$
determines the nature of the differential $d^4_{4,0}=\times 2$ that is multiplication by 2, which forces 
$E^\infty_{4,0}=\tn{Ker}\,d^4_{4,0}=0$. That was the only missing part of the limit page. Summing up the south-east diagonals in the limit page gives the answer.
\end{proof}

\ni Next, using the information coming out of this proposition, we are going to manage a higher Stiefel manifold $V_3\mbbr^6$. We have the following result on this manifold.

\begin{prop}\label{homologyV3R6}
The homology of the Stiefel manifold $V_3\mbbr^6$ is the following. 
$$ H_*(V_3\mathbb{R}^6;\mbbz)=(\mathbb{Z},0,0,\mathbb{Z}_2,0,\mathbb{Z},0,\mbbz,\mbbz_2,0,0,0,\mathbb{Z})$$ \end{prop}
\begin{proof}
The Stiefel manifold $V_3\mbbr^6$ is by definition equal to the set of 3-frames in 6-space.  Projection onto the first vector gives the following fiber bundle  (\ref{V2R5V2R6fibration}) with fiber $V_2\mbbr^5$. 
\begin{equation}
V_2\mbbr^5 \rightarrow V_3\mbbr^6 \longrightarrow \mbb S^5
\label{V2R5V2R6fibration}\end{equation}
Although we can use the homology version as well, we are going to use the cohomological Serre spectral sequence related to this fiber bundle (\ref{V2R5V2R6fibration}) to be able to use it for cup product calculations as well.  We define it with the description of its limit as follows.
$$E^{p,q}_2:=H^p(\mbb S^5;H^q(V_2\mbbr^5;\mathbb{Z}))$$  $$E_\infty^{p,q} = F^{p,q}/{F^{p+1,q-1}}$$
\ni where abelian groups $F^{p,q}$ form a  filtration that satisfies 
\begin{equation} H^n(V;\mbbz) =F^{0,n}\supset F^{1,n-1}\supset \cdots \supset
F^{n+1,-1}=0. \label{filtrationcohomology}\end{equation}
This sequence behave appropriately because the base manifold is simply connected.  
\ni There exists also homomorphisms called the differential maps $d_n^{p,q}$ such that  $$d_n^{p,q}:E_n^{p,q}\rightarrow E_n^{p+n,q-n+1}$$
This spectral sequence 
converges immediately and illustrated on the Table \ref{table:V3R6serre}.  
\beg{table}[ht] 
\caption{ {\em Cohomological Serre spectral sequence for $V_3\mbb R^6$. 
}}
\bct{
$\begin{array}{cc|c|c|c|@{} c @{}|c|c|} \cline{3-8}
    &7& \mbbz &       &   &&&\mbbz\\ \cline{3-8}
    &6&       &       &   &&&\\ \cline{3-8}
    &5&       &       &   &&&\\ \cline{3-8}
    &4&\mbbz_2&       &   &&&\mbbz_2\\ \cline{3-8}
    &3&       &       &   &&&\\ \cline{3-8}
    &2&       &       &   &&&\\ \cline{3-8}
    &1&       &       &   &&&\\ \cline{3-8}
E_2=E_\infty~~~&0&\mbbz&       &   &&&\mbbz  \\ \cline{3-8}
& \multicolumn{1}{c}{ } & \multicolumn{1}{c}{0} 
& \multicolumn{1}{c}{1} & \multicolumn{1}{c}{2} 
& \multicolumn{1}{c}{3} & \multicolumn{1}{c}{4}
& \multicolumn{1}{c}{5}   \\  
\end{array}$   
\setlength{\unitlength}{1mm}
\begin{picture}(0,0)
\put(-35,4.5){\vector(3,-2){26}}   
\end{picture} } \ect \label{table:V3R6serre}
\end{table}
\ni The only potentially non-trivial differential is the following. 
$$d^{0,4}_5 : \mbbz_2 \longrightarrow \mbbz$$
which is zero because it maps torsion to a free space. So that we have the limit as well at from the beginning. Accumulating the groups in the south-east direction again and using Poincar\'e duality yields the result. 	 \end{proof} 

\ni Next, we are going to compute the cup products, before which we need a Lemma. 
\begin{lem}Considering the fiber bundle (\ref{V2R5V2R6fibration}) the following pullback maps  
$$H^*(V_2\mbbr^5;\mbbz) \longleftarrow H^*(V_3\mbbr^6;\mbbz):i^*$$
$$H^*(V_3\mbbr^6;  \mbbz) \longleftarrow H^*(\mbb S^5;\mbbz):\pi^*$$
induced by an embedding $i:V_2\mbbr^5 \rightarrow V_3\mbbr^6$ 
onto some fixed fiber over a point of the base and 
the projection $\pi: V_3\mbbr^6 \rightarrow \mbb S^5$ of the total space are surjective and injective, respectively.
\label{lerayhirschhypothesis}\end{lem}
\begin{proof} These pullback maps correspond to the following \cite{hajimesato} natural maps and compositions in the spectral sequence. 
$$\hat{i}: H^n\!(V_3^6;\mbbz)\!=\!F^{0,n} \longrightarrow F^{0,n}\!\! /\! F^{1,n-1}\!\!=\!E^{0,n}_\infty\!\!=\!E^{0,n}_2\!\!=\!H^0\!(\mbb S^5;H^n\!(V_2^5;\mbbz))\!\approx\! H^n\!(V_2^5;\mbbz)$$
$$\hat\pi:H^n(\mbb S^5;\mbbz)\approx E^{n,0}_2=E^{n,0}_\infty\approx 
F^{n,0}\!/F^{n+1,-1}\approx F^{n,0}\longrightarrow F^{0,n}=H^n(V_3\mbbr^6;\mbbz)$$ Realize that the first map is a quotient map other than the identifications so that it is surjective. The second map is an inclusion in the filtration (\ref{filtrationcohomology}) other than the identifications, hence an injection.  
\end{proof}

\ni We can summarize the Propositions \ref{homologyV2R5} and \ref{homologyV3R6} in terms of cohomology as follows. 
\begin{cor}\label{Stiefelring}
The cohomology ring of the Stiefel manifolds are the following for which $\tn{deg}x_m=m$. 
\begin{enumerate}
\item $H^*(V_2\mbbr^5;\mbbz)=\mbbz[x_7]/(x_7^2)\oplus\mbbz_2[x_4]/(x_4^2)$
\item $H^*(V_3\mbbr^6;\mbbz)=\mbbz[x_5,x_7]/(x_5^2,x_7^2)\oplus\mbbz_2[x_4,x_9]/(x_4^2,x_9^2,x_4x_9)$
\end{enumerate}
\end{cor}
\begin{proof} The first ring is obtained out of the dimensional restrictions. To deal 
with the second ring, after the dimensional regulations, we finally have to determine the fate of the top dimensional graded element 
$$x_5x_7\in H^{12}(V_3\mbbr^6;\mbbz).$$ 
For this purpose we will use the Leray-Hirsch theorem \cite{hatcherat,spanier} 
which gives an isomorphism on the rational cohomology. Only crucial hypothesis is what we to proved in the previous Lemma \ref{lerayhirschhypothesis} that the 
cohomological pullback map from the total space to the fiber of the fibration 
$i^*=\hat i$ is a surjective map. Then by the theorem we have an isomorphism 
between the product of fiber and the base and the total space as follows. 
$$H^*(V_2\mbbr^5;\mbb Q)\otimes H^*(\mbb S^5;\mbb Q) 
\stackrel{\sim}\longrightarrow H^*(V_3\mbbr^6;\mbb Q)$$
\ni Restricting this isomorphism to the 12-th grading gives us the isomorphism 
$$H^7(V_2\mbbr^5;\mbb Q)\otimes H^5(\mbb S^5;\mbb Q) 
\stackrel{\sim}\longrightarrow H^{12}(V_3\mbbr^6;\mbb Q)$$
which means that our product has to be a generator. 
\end{proof}

\section{Some Homotopy Theory}\label{sechomotopy}
In this section we are going to compute some of the homotopy theoretic  invariants of the Grassmannian manifold. 
We need the following fiber bundle to find homotopy groups of the Grassmannian.
\begin{equation}\label{g3r6stiefelfibration}
\mbb{SO}_3 \to V_3\mbbr^6 \longrightarrow G_3^+\mbbr^6 \end{equation}

\ni Exploiting this fibration,  we lead to the following result. 

\begin{lem}\label{lempi3} The preliminary homotopy groups of the Grassmannian are the following.  $$\pi_{01234}\, G_3^+\mbbr^6=(0,0,\mbbz_2,\mbbz_2,\mbbz).$$ 
\end{lem}
\begin{proof}
We apply the homotopy exact sequence to the fiber bundle (\ref{g3r6stiefelfibration}), a part of which is as follows. 
\begin{eqnarray}\label{V3R6hes}\nonumber
&\cdots\to\pi_5\, G_3^+\mbbr^6 \to 
\mbbz_2 \to\mbbz_2\to \pi_4\, G_3^+\mbbr^6 \to 
\mbbz   \to\mbbz_2\stackrel{\pi_*}\twoheadrightarrow \pi_3\, G_3^+\mbbr^6 \to &\\ 
&      0 \to    0  \to \pi_2\, G_3^+\mbbr^6 \to 
\mbbz_2 \to 0 \rightarrow \pi_1\, G_3^+\mbbr^6 \to 0.&
\end{eqnarray}
\ni We use the fact that $V_3\mbbr^6$ is 2-connected and Proposition 
\ref{homologyV3R6} for the Stiefel manifold, and higher 
homotopy groups of $\mbb{SO}_3$ is the same as of its universal cover which is 
the 3-sphere. Surjectivity of the map $\pi_*$ reveals that the only option 
for the 3rd level is $\mbbz_2$ other than the trivial group. Since the Grassmannian is 1-connected, the Hurewicz homomorphism    
$$h: \pi_3\,G_3^+\mbbr^6 \longrightarrow H_3(G_3^+\mbbr^6;\mbbz)$$
\ni is an epimorphism by \cite{hatcherat} at this level. This implies that 
$\mbbz_2$ is the only nontrivial option for the 3rd homology as well. 
One can continue to analyse this exact sequence by inserting the homotopy 
groups of the Stiefel manifold from \cite{paechter1956} starting from the 4-th level as we did above. 
\vspace{.05in}

We need to work with another fibration involving special orthogonal groups which is used to define the Grassmannian as well, 
\begin{equation}\label{G3R3definitionfibration}
\mbb{SO}_3\times \mbb{SO}_3 \to \mbb{SO}_6 \longrightarrow 
G_3^+\mbbr^6 \end{equation}

\ni The homotopy sequence of this fibration at 
the 3rd level reads as the following,

$$0 \to \mbbz \to \mbbz\oplus\mbbz \stackrel{i_*}\to \mbbz 
\to \pi_3\,G_3^+\mbbr^6 \to 0.$$

\ni The homotopy groups of the special 
orthogonal groups can be deduced from 
orthogonal fibrations, and 
$\pi_4\,G_3^+\mbbr^6$ injects into 
$\mbbz\oplus\mbbz$ which makes it a subgroup of
 a free group hence itself free. This with (\ref{V3R6hes}) resolves the 4-th level. 
Cokernel has to be $\mbbz$ or $2\mbbz$. 
After this we can turn the problem into matrices. The 
group $\mbb{SO}_6$ is a smooth 15-dimensional manifold. According to the standard embedding, in its 3x3 block lies a copy of $\mbb{SO}_3=\mbb{RP}^3$. 
We need to understand the following map between the integers. 
\begin{equation} i_* : \pi_3 (\mbb{SO}_3) \stackrel{\times 2}\longrightarrow \pi_3(\mbb{SO}_6).\label{3to6multby2}    
\end{equation}
\ni So, up to a sign, this homomorphism either maps the  generator to a generator or maps to twice the generator of 3rd homotopy of $\mbb{SO}_6$. 
We claim that this map is multiplication by 2 
upto sign, so that the embedded $\mbb{SO}_3$ does not
generate the 3rd homotopy group 
of $\mbb{SO}_6$, rather only the even members.  
To prove this statement let us first 
understand the map, 
\begin{equation} 
i_* : \pi_3(\mbb{SO}_3)\longrightarrow\pi_3(\mbb{SO}_4)
\label{3to4multby2} 
\end{equation}
\ni through using the fiber bundle of Stiefel manifolds, 
\begin{equation*}
\mbb{SO}_3 \to \mbb{SO}_4 \longrightarrow \mbb S^3.
\end{equation*}
\ni The homotopy exact sequence at the 3rd level reveals the following, 
{\renewcommand*{\arraystretch}{1.5}
\begin{equation}
\begin{array}{ccccccccc}  
\cdots&\to&\pi_3(\mbb{SO}_3)&  \stackrel{i_*}\to &\pi_3(\mbb{SO}_4) &\stackrel{e}\to  &\pi_3(\mbb S^3) 
&\to&\cdots \\ 
&&&&\downarrow \overset{}{\frac{p_1}{2}}&&&&\\ 
&&&& \mbbz &&&&  
\end{array}    \label{ponteuler}
\end{equation}
}
\ni Here, $\pi_3(\mbb{SO}_4)\approx\mbbz\oplus\mbbz$ parametrizes the real 4-bundles or equivalently $\mbb S^3$-bundles over $\mbb S^4$ as follows. Considering the atlas consisting of the two charts produced by taking out the north or the south pole. Consider the local trivializations 
$$\varphi_N:\pi^{-1}(\mbb S^4-N)\to(\mbb S^4-N)\times\mbbr^4,~~~\varphi_S:\pi^{-1}(\mbb S^4-S)\to(\mbb S^4-S)\times\mbbr^4 $$
projecting to these contractible spaces. The transition function restricted to the above 
of the equivator is 
$$\varphi_S\circ\varphi_N^{-1}|_{\mbb S^3\times\mbbr^4}: \mbb S^3\times\mbbr^4\to \mbb S^3\times\mbbr^4$$
\ni which acts as $(p,u)\mapsto(p,g(p)u)$ for a function
$g: \mbb S^3\to \mbb{SO}_4$ called the {\em clutching} function. 
Free homotopy type of the clutching functions classify 4-bundles 
upto isomorphism. $\mbb{SO}_4$ is doubly covered by $\mbb 
S^3\times\mbb S^3$ via the conjugation map of unit quaternions. 
Invariance of higher homotopy groups under covering projections 
imply the above equivalence. 
Alternatively, the homeomorphism $\mbb{SO}_4\approx \mbb 
S^3\times \mbb{SO}_3$ can also be used if one does not care for 
the group structure. 
Fixing a vector $u\in \mbb S^3$ and defining a map $f:\mbb 
S^3\to\mbb S^3$ by $f(x)=g(x)u$ we define the {\em Euler} class 
$e$ of the bundle defined by the degree of the map $f$. 
So this is the way to define the map from middle to the right in 
the above homotopy sequence. 
By \cite{walschapeuler} upto a sign, this is consistent with the 
definition of the Euler class as a characteristic class. 

On the other hand, the vertical map in (\ref{ponteuler}) is 
surjective, for example  because of the 7-spheres of 
\cite{milnors7spheres}. 
One can show that the Euler class for the 
quaternionic Hopf bundle $H$ 
on the 4-sphere  is the generator though  the Euler class for the tangent space $T$ 
 is the twice of the 
generator, see \cite{braddellthesis} for an explanation.  
So that $e(T-2H)=0$. 
Since the kernel is free and the 
map is non trivial, this element generates the kernel, hence the 
image of $i_*$. 

\vspace{.05in}

The 3rd homotopy group of the special orthogonal groups, 
$$\pi_3\,\mbb{SO}_n =\left\{\begin{array}{cc}
 0    & n = 2 \\ 
\mbbz\oplus\mbbz & n = 4   \\ 
\mbbz  & n=3, ~  n\geq 5
\end{array}   \right.$$
\ni reveals the fact that the half-Pontrjagin map from the stable homotopy group,
$${p_1}/{2} : \pi_3\,\mbb{SO} \longrightarrow \mbbz$$ 
\ni gives an isomorphism with integers. 
The Pontrjagin classes for spheres are trivial although the total
Pontrjagin class for the quaternionic bundle is $p(H)=1-2x_4$. So $$(p_1/2)(T-2H)=(0+4x_4)/2=2x_4$$
\ni which generates the image of $i_*$ of (\ref{3to6multby2}) and
because of the stability 
generates the image of $i_*$ of (\ref{3to4multby2}).
\end{proof}

\begin{rmk}
We have found the fourth homotopy group from the above fibration (\ref{G3R3definitionfibration}). In \cite{gluckmackenziemorgan} some of the rational homotopy grous are computed and the fourth level generator is the PONT manifold which is calibrated by the first Pontrjagin form.  
\end{rmk}

\newpage

\section{The Grassmannian Manifold $G_3^+\mathbb R^6$}\label{secgrassmann}
In this section we are going to compute 
other algebraic topological  invariants of the Grassmannian manifold. 
Akbulut and Kalafat computed \cite{atg2} homology groups of 
various Grassmann bundles. We will use similar techniques  to 
compute homology groups of  $G_{3}^{+}\mathbb{R}^{6}.$ 
At the preliminary levels we can use the 
homotopy theory in the previous section.   
Lemma \ref{lempi3}
combined with the Hurewicz isomorphism we obtain, 
$$\mbbz_2=\pi_2\, G_3^+\mbbr^6 \approx H_2(G_3^+\mbbr^6;\mbbz).$$ 
\ni We continue with the following central preliminary result. 

\begin{lem} \label{nontrivial} The homology group $H_3(G_3^+\mbb R^6;\mbb Z)$ is trivial.
\end{lem}
\begin{proof}
In order 
to prove our result, we assume that this homology groups is nontrivial. 
So far, after making this assumption, we are able to handle the homology upto the third level. The free homology is obtained through the Poincar\'{e} polynomial \cite{ghv,gluckmackenziemorgan}
\begin{equation}p_{G_3^+\mbbr^6}(t)=(1+t^4)(1+t^5).
\label{poincarepolynomial}\end{equation}
\ni Various forms of the universal coefficients theorem \cite{hajimesato} coupled with the Poincar\'{e} duality may be used to compute the higher level torsion as follows.
\begin{eqnarray}
T_5 & = & F_6\oplus T_5  \nonumber\\ [2\jot]
 & = &\tn{Hom}(H_6,\mbbz)\oplus \tn{Ext}(H_5,\mbbz) \nonumber
 \\ [2\jot]  
 & = & H^6(G;\mbbz) = H_3(G;\mbbz)=\mbbz_2. \nonumber 
\end{eqnarray}
\ni Similarly $T_6=H_2(G;\mbbz)=\mbbz_2$ and 
$T_7=H_1(G;\mbbz)=0$. Moreover, from comparing the torsion parts of the equality for the next case,
$$\mbbz\oplus T_8=H_0(G;\mbbz)=\mbbz$$ \ni we get $T_8=0$. 
Collecting the results obtained from the assumption, 
the  homology groups of the oriented Grassmann manifold $G_{3}^{+}\mathbb{R}^{6}$ hypothetically has to be the following sequence.  
	\begin{equation*}\mathcal H_*(G_3^+\mathbb{R}^{6};\mbbz)=(\mathbb{Z},0,\mathbb{Z}_2,\mathbb{Z}_2,\mathbb{Z}\oplus T_4,\mathbb{Z}\oplus\mbbz_2,\mbbz_2,0,0,\mathbb{Z}).
	\end{equation*}
\ni To be able to raise a contradiction, 
we need to define the cohomological Serre spectral sequence related to the fiber bundle (\ref{g3r6stiefelfibration}) with the limit as follows.
\begin{equation*}
\nonumber
E^{p,q}_2:=H^p(G_{3}^{+} \mathbb{R}^{6};H^q(SO_3;\mathbb{Z})) 
\end{equation*}
\begin{equation}\label{V3R6cohomologicalsss} 
E_\infty^{p,q} = F^{p,q}/{F^{p+1,q-1}}.
\end{equation}
\ni where abelian groups $F^{p,q}$ form a  filtration that satisfies 
\begin{equation}\label{filtrationsss}
H^n(V;\mbbz) =F^{0,n}\supset F^{1,n-1}\supset \cdots \supset
F^{n+1,-1}=0. \end{equation}
This sequence behave appropriately because the base manifold is simply connected.  
\ni There exists also homomorphisms called the differential maps $d_n^{p,q}$ such that  $$d_n^{p,q}:E_n^{p,q}\rightarrow E_n^{p+n,q-n+1}$$
Then keeping in mind for the fiber that, 	
\begin{equation}H_*(\mbb{SO}_3;\mathbb{Z})=(\mathbb{Z},\mathbb{Z}_2,0,\mathbb{Z}) \label{homologySO3}
\end{equation}
\ni we figure out the 
second page of the spectral sequence as can be seen on Table \ref{table:cohomsssG3R6hypo}.

\begin{table}[ht] 
\caption{ {\em Hypothetical cohomological Serre spectral sequence for $G_3^+\mbb R^6$ 
}}
\bct{\large
$\begin{array}{cc|c|c|c| c |@{}c@{}|@{}c@{}|c|c|c|c|} \cline{3-12}
  &3&\mbbz &~0~&0&\mbbz_2&\mbbz\oplus\mbbz_2& \mbbz\oplus T_4 &\mbbz_2&\mbbz_2&~0~& \mbbz\\ \cline{3-12}
  &2&\mbbz_2&0&\mbbz_2&\mbbz_2^2&H_5(G;\mbbz_2)&H_4(G;\mbbz_2)&\mbbz_2^2&\mbbz_2&0&\mbbz_2\\ \cline{3-12}
  &1&       &&&&&         &&&&  \\ \cline{3-12}
E_2~~~
  &0& \mbbz &0&0&\mbbz_2&\mbbz\oplus\mbbz_2& \mbbz\oplus T_4 &\mbbz_2&\mbbz_2&0& \mbbz\\ \cline{3-12}
& \multicolumn{1}{c}{ } & \multicolumn{1}{c}{0} 
& \multicolumn{1}{c}{1} & \multicolumn{1}{c}{2} 
& \multicolumn{1}{c}{3} & \multicolumn{1}{c}{4}
& \multicolumn{1}{c}{5} & \multicolumn{1}{c}{6} 
& \multicolumn{1}{c}{7} & \multicolumn{1}{c}{8} 
& \multicolumn{1}{c}{9} 
\\  
\end{array}$   
\setlength{\unitlength}{1mm}
\begin{picture}(0,0)
\put(-22,6){\vector(3,-1){12}}
\end{picture}
}\ect \label{table:cohomsssG3R6hypo}
\end{table}
\ni Because of the vanishing of the cohomology of the Stiefel manifold, the filtration at the 6-th level degenerates.
$$0=H^6(V_3^6;\mbbz)=F^{0,6}\supset F^{1,5}\supset\cdots 0.$$
\ni So that the vanishing of the term $F^{4,2}$ forces the vanishing of the limit $E_\infty^{4,2}$. This is possible by a trivial kerneli 
so through the injectivity of the differential,
$$d_3^{4,2}: E_3^{4,2}\approx E_2^{4,2} \hookrightarrow \mbbz_2.$$
\ni On the other hand, the domain of this differential can be computed by 
\begin{eqnarray}
E_2^{4,2} & = & H^4(G;\mbbz_2)  \nonumber\\ [2\jot]
 & = &\tn{Hom}(H_5,\mbbz_2)\oplus \tn{Ext}(H_4,\mbbz_2) \nonumber\\ [2\jot]
 & = & \mbbz_2\oplus\mbbz_2\oplus\tn{Ext}(H_4,\mbbz_2). \nonumber 
\end{eqnarray}

\ni which evidently can not inject into $\mbbz_2$.  \end{proof}

Now, we are ready to prove the main result of this section. 
By further applications of 
Serre spectral sequence we obtain the following list.

\begin{thm}
	The homology of the oriented Grassmann manifold $G_{3}^{+}\mathbb{R}^{6}$ is given by  \begin{equation*}
H_*(G_3^+\mbbr^6;\mbbz)=(\mbbz,0,\mbbz_2, 0 ,
\mathbb{Z}, \mbbz ,\mbbz_2,0,0,\mathbb{Z})  \end{equation*}
\label{homologyG3+R6}\end{thm}

\begin{proof} Using the homotopy and Hurewicz theorem together with Lemma \ref{nontrivial}, 
we handle the homology upto the third level. The free homology is obtained through the Poincar\'{e} polynomial (\ref{poincarepolynomial}).  
\ni Various forms of the universal coefficients theorem \cite{hajimesato} coupled with the Poincar\'{e} duality may be used to compute the higher level torsion as follows.
\begin{eqnarray}
T_5 & = & F_6\oplus T_5  \nonumber\\ [2\jot]
 & = &\tn{Hom}(H_6,\mbbz)\oplus \tn{Ext}(H_5,\mbbz) \nonumber\label{variablex2k-2}\\ [2\jot]  
 & = & H^6(G;\mbbz) = H_3(G;\mbbz)=0. \nonumber \label{variabley2k-2}
\end{eqnarray}
\ni Similarly $T_6=H_2(G;\mbbz)=\mbbz_2$ and 
$T_7=H_1(G;\mbbz)=0$. Moreover, from comparing the torsion parts of the equality for the next case,
$$\mbbz\oplus T_8=H_0(G;\mbbz)=\mbbz$$ \ni we get $T_8=0$. Hence the 4-th level torsion subgroup $T_4$ is the only remaining case. To handle this case we again define the cohomological Serre spectral sequence (\ref{V3R6cohomologicalsss}) related to the fiber bundle (\ref{g3r6stiefelfibration}). 
This sequence behave appropriately because the base manifold is simply connected.  
After various applications of the universal coefficients theorem like, 
\begin{equation*}
E_2^{8,2}=H^8(G;\mbbz_2) = H^8(G;\mbbz)\otimes\mbbz_2\oplus \tn{Tor}(H^9,\mbbz_2)=0 \end{equation*}
we figure out the 
second page of the spectral sequence as can be seen on Table \ref{table:cohomsssG3R6}.

\begin{table}[ht] 
\caption{ {\em Cohomological Serre spectral sequence for $G_3^+\mbb R^6$ 
}}
\bct{\large
$\begin{array}{cc|c|c|c| c |@{}c@{}|@{}c@{}|c|c|c|c|} \cline{3-12}
  &3&\mbbz &~0~&0&\mbbz_2&\mbbz& \mbbz\oplus T_4 & 0 &\mbbz_2&~0~& \mbbz\\ \cline{3-12}
  &2&\mbbz_2&0&\mbbz_2&\mbbz_2&\mbbz_2\oplus T_4&\mbbz_2\oplus T_4&\mbbz_2&\mbbz_2&0&\mbbz_2\\ \cline{3-12}
  &1&       &&&&&         &&&&  \\ \cline{3-12}
E_2~~~
  &0& \mbbz &0&0&\mbbz_2&\mbbz& \mbbz\oplus T_4 & 0 &\mbbz_2&0& \mbbz\\ \cline{3-12}
& \multicolumn{1}{c}{ } & \multicolumn{1}{c}{0} 
& \multicolumn{1}{c}{1} & \multicolumn{1}{c}{2} 
& \multicolumn{1}{c}{3} & \multicolumn{1}{c}{4}
& \multicolumn{1}{c}{5} & \multicolumn{1}{c}{6} 
& \multicolumn{1}{c}{7} & \multicolumn{1}{c}{8} 
& \multicolumn{1}{c}{9} 
\\  
\end{array}$   
\setlength{\unitlength}{1mm}
\begin{picture}(0,0)
\put(-22,6){\vector(3,-1){12}}
\end{picture}
}\ect \label{table:cohomsssG3R6}
\end{table}
\ni The only missing part is the torsion at the fourth level. 
To address this problem we 
analyse the filtration at the 6-th level, so that we have the following, 
$$d_3^{4,2}: \mbbz_2\oplus T_4 \hookrightarrow \mbbz_2.$$
 The differential $d_3^{4,2}$ is injective because of the vanishing of the 6-th diagonal 
$$0=H^6(V_3^6;\mbbz)=F^{0,6}\supset F^{1,5}\cdots \supset F^{4,2}$$
\ni which forces the group $T_4$ to vanish. 
\end{proof}
\ni Alternatively, one can work with the defining standard fibration (\ref{G3R3definitionfibration}) involving special orthogonal groups to obtain similar topological information.  

\begin{rmk}
Another alternative is the following useful double fibration. 
$$\xymatrix{
           & S^3\ar[d]& \\
S^2 \ar[r] & S(E_3\mbb R^6)=E_0 \ar[d]^{\pi_v}\ar[r]^{~~~~~~\pi_h} &G_3^+\mbb R^6\\
  &G_2^+\mbb R^6 &  }$$ 


\noindent The tautological bundle over 
$G_3^+\mbb R^6$ is denoted by $E_3\mbb R^6$. The 2-sphere bundle $S(E_3\mbb R^6)$
is obtained by furnishing a with a metric and taking the unit sphere in each fiber. That is how the horizontal fibration obtained. 
From there, one can obtain the vertical fibration
with the following procedure. A point in $G_2^+\mbb R^6$ represents a 2-
plane which can be extended to a 3-plane by adding a unit vector in the 6-2=4 
dimensional orthogonal complement. Another interpretation is through the oriented flag variety $F_{2,3}^+(\mbb R^6)$ with its projection maps. 
On these vertical and horizontal fibrations the Gysin exact sequence \cite{milnor} 
can be used to produce information.
\end{rmk}


\section{The special Lagrangian locus of 
$G_3^+\mathbb{R}^6$
}\label{secslag}

In this section we calculate some invariants of the 
submanifold SLAG. Our techniques are entirely different 
than that of \cite{gluckmackenziemorgan}.  
We start with some definitions. 

\begin{defn} Let $(M,J,g,\Omega)$ be a compact \textit{Calabi-Yau manifold} 
with complex dimension $m$ and holonomy $\mbb{SU}_m$ where 
$\Omega=\Omega_1+i\Omega_2$ is a holomorphic volume form. The real part $Re(\Omega)$ is 
calibration on $M,$ and the corresponding calibrated submanifolds are called
\textit{special Lagrangian submanifolds}. 
\end{defn}
Then, the contact set in $G_{3}^{+}\mathbb{R}^{6}$ is denoted by $SLAG$ with the given calibration $Re(\Omega)$ in $\mathbb{R}^{6}.$ 
SLAG is a compact 5-dimensional space. 
The homology of $SLAG$ is the following; 

\begin{thm}	The homology of the oriented SLAG manifold in $G_{3}^{+}\mathbb{R}^{6}$ is given by:  
\begin{equation*} H_{\ast }(SLAG;\mathbb{Z})=(\mathbb{Z},0,\mathbb{Z}_2,0, 0,\mathbb{Z})	\end{equation*}
\end{thm}

\begin{proof} $\mbb{SU}_3$ as a subset of $\mbb{SO}_6$ acts on $\mbbr^6$ by matrix multiplication, so that it acts on the oriented 3-planes as well. This action gives the following fiber bundle. 
\begin{equation}\mbb{SO}_3\to \mbb{SU}_3 \longrightarrow SLAG \label{SLAGfibration}\end{equation}
Since the holomorphic 3-form is $\mbb{SU}_3$-invariant, special Lagrangian 3-planes are also invariant, hence the action. The stabilizer is the special orthogonal group acting on the complement of the 3-plane. The homotopy exact sequence on the fiber bundle (\ref{SLAGfibration}) reveals that the 5-manifold $SLAG$ is simply connected as well as the unitary group $\mbb{SU}_3$. So that the Serre spectral sequence would behave reasonably. We prefer to work with the homological version as follows. 
$$E_{p,q}^2:=(SLAG;H_q(\mbb{SO}_3;\mathbb{Z}))$$ 
$$E^\infty_{p,q} = F_{p,q}/{F_{p-1,q+1}}$$ 
\ni The invariants of the 
fiber is the same as in \ref{homologySO3}, however for the total space we have the following.
	$$H_*(\mbb{SU}_3;\mathbb{Z})=(\mathbb{Z},0,0,\mathbb{Z},0,\mathbb{Z},0,0,\mathbb{Z})$$
Then, the Table \ref{table:SLAG} illustrates the second page of the spectral sequence. 
\beg{table}[ht] \caption{ {\em Homological Serre spectral sequence for $SLAG$.
}}
\bct{\large
$\begin{array}{cc|c|c|@{} c @{}|@{} c @{}|c|c|} \cline{3-8}
    &3&\mbbz& ~~~~ & H_2(SLAG;\mathbb{Z}) & H_3(SLAG;\mathbb{Z}) &~~~~&\mbbz\\ \cline{3-8}
    &2&       &       &&&&\\ \cline{3-8}
    &1&\mbbz_2& & H_2(SLAG;\mathbb{Z}_2)&H_3(SLAG;\mathbb{Z}_2)& &\mbbz_2\\ \cline{3-8}
E^2~~~&0&\mbbz& & H_2(SLAG;\mathbb{Z}) & H_3(SLAG;\mathbb{Z}) &  &\mbbz\\ \cline{3-8}
& \multicolumn{1}{c}{ } & \multicolumn{1}{c}{0} 
& \multicolumn{1}{c}{1} & \multicolumn{1}{c}{2} 
& \multicolumn{1}{c}{3} & \multicolumn{1}{c}{4}
& \multicolumn{1}{c}{5}\\  
\end{array}$   
\setlength{\unitlength}{1mm}
\begin{picture}(0,0)
\put(-86,-4){\vector(-3,1){12}} 
\end{picture} }
\ect
\label{table:SLAG} \end{table}

 \ni	Now, let us explain the entries of this table. Vanishing of the 1-st and 4-th columns are provided by simple connectivity and the universal coefficients. Another application of the universal coefficient theorem gives the following. 
$$H_4(SLAG;\mbbz_2)=\tn{Hom}(H^4,\mbbz_2)\oplus\tn{Ext}(H^5,\mbbz_2)=0.$$
\ni The other entries can be handled in a similar manner. To find out the second homology we proceed by analysing the limit. According to the filtration (\ref{filtrationhomology}) we have $$F_{0,1}\subset F_{1,0}=H_1(\mbb{SU}_3;\mbbz)=0$$ 
hence the filtration vanishes identically. So that the quotients and the limit $E_{1,0}^\infty=F_{1,0}/F_{0,1}$. Entry here have to vanish eventually. This implies that the following differential is surjective. 
\begin{equation}\mbbz_2 \longleftarrow H_2(SLAG;\mbbz):d_{2,0}^2\label{differential}\end{equation}
We are able to do a similar kind of analysis in the next diagonal as well. Vanishing of the second homology of the special unitary group implies that the filtration
$$F_{0,2}\subset F_{1,1}\subset F_{2,0}=H_2(\mbb{SU}_3;\mbbz)=0$$
vanish identically, as well as the quotient $E_{2,0}^\infty=F_{2,0}/F_{1,1}$. 
To provide that, the differential (\ref{differential}) has to be injective hence an isomorphism, which implies that the second homology $H_2(SLAG;\mbbz)=\mbbz_2$.

\vspace{.05in}

To figure out the latest unknown, in the table we notice that $E^2_{3,3}=E^\infty_{3,3}$ because of the differentials pointing at and from zero. 
Vanishing of the homology of special unitary group at this level again implies the filtration 
$$F_{2,4}\subset F_{3,3}\subset F_{4,2}\subset F_{5,1}\subset F_{6,0}=H_6(\mbb{SU}_3;\mbbz)=0$$
to vanish identically, which effects the limit as the following, 
$$H_3(SLAG;\mbbz)=E^\infty_{3,3}=F_{3,3}/F_{2,4}=0.$$\end{proof}

\ni Consequently we get the following ring structure. 

\begin{cor}\label{SLAGring} The cohomology ring of the SLAG manifold is the following 
truncated polynomial ring for which $\tn{deg}x_m=m$.
\begin{eqnarray*}
H^*(SLAG;\mbbz) 
 & = & \mbbz[x_3,x_5]/(2x_3,x_3^2,x_5^2,x_3x_5)\\ [2\jot]
 & = & \mbbz[x_5]/(x_5^2)\oplus\mbbz_2[x_3]/(x_3^2). \nonumber
\end{eqnarray*}
\end{cor}

\section{The Ring Structure}\label{secring}
In this section we are going to compute the cohomology ring 
of the Grassmann manifold. The cohomological Serre spectral 
sequence is more appropriate for this task.
In the following discussion, $E=E(3,6)$ will denote 
the canonical bundle over the Grassmann manifold 
$G_3^+\mbbr^6$. It is obtained by taking the plane corresponding to a point 
to produce a vector bundle of rank 3 over our Grassmannian. 
We collect the related results of \cite{shizhou} here as follows.

\begin{thm}\label{shizhouthm}
We have the following relations in the cohomology of 
the Grassmann manifold $G_3^+\mbbr^6$.\\ 

(a) $2^{-1}p_1E(3,6)$ is a generator for 
$H^4(G_3^+\mbbr^6;\mbbz)$ and its Poincar\'{e} dual is 
$[SLAG]$ 

which is a generator of $H_5(G_3^+\mbbr^6;\mbbz)$.\\

(b) $4(3\pi)^{-1}*p_1E(3,6)$ is a generator for 
$H^5(G_3^+\mbbr^6;\mbbz)$ and its Poincar\'{e} dual is 

$[G(2,4)]$ homologous to $[PONT]$ 
which is a generator of $H_4(G_3^+\mbbr^6;\mbbz)$.
\end{thm}

\ni Using these characteristic classes and integrals we will be able to 
figure out the generators and relations in our Grassmannian. 
Now we are ready to compute the cohomology ring. 

\begin{thm}\label{G3R6ring}
The cohomology ring of the Grassmannian $G_3^+\mbbr^6$ is as 
follows where $\tn{deg}x_m=\tn{deg}y_m=m$.
$$\hspace{-6mm}H^*(G_3^+\mbbr^6;\mbbz)=
\mbbz[x_4,x_5]/\langle x_4^2,x_5^2, x_4x_5-x_5x_4 \rangle \oplus
\mbbz_2[y_3,y_7]/\langle y_3^2,y_7^2,y_3y_7  
\rangle.$$
\end{thm}
\begin{proof}  We start with the free part which requires some 
attention. We pick the generators in the levels 4 and 5 and assign the 
following $x_4:=2^{-1}p_1E(3,6), x_5:=4(3\pi)^{-1}*p_1E(3,6)$   
values. After handling the relations produced by the dimensional restrictions we have to understand the product of the chosen generators. 
So that we evaluate $x_4x_5$ over the Grassmannian to get its coefficient. 
The integral, 
\begin{align*} _{G_3^+\mbbr^6}\int 
x_4 x_5 = _{G(2,4)}\int x_4 =1\label{relation1}\end{align*}
provides that the product $x_4x_5$ is the generator of the top level. 
\end{proof}
\ni Alternatively we suggest to the reader to compute the product structure through the cohomological Serre spectral sequence. 

\vspace{.05in}

Next, we would like to see how the submanifold of special Lagrangian 3-planes 
sits inside the Grassmannian $G_3^+\mbbr^6$ cohomologically. We have the pullback map induced by the inclusion
$$H^q(SLAG;\mbbz) \longleftarrow H^q(G^+_
3\mbbr^7;\mbbz) : i^*$$
\ni which is trivial other than the levels $q = 0 \cdots 5$. Since $SLAG$ 
has nontrivial cohomology only at the levels $3$ and $5$, we 
only need to understand the pullback map at those 
levels. 
We can work on the 4-manifold $PONT$ in a similar way. 
Considering the cohomology ring of Pontrjagin cycles, 
$$H^*(PONT;\mbbz)=\mbbz[x_2,y_2]/\langle x_2^2,y_2^2,x_2y_2-y_2x_2\rangle$$
\ni we have the following.

\begin{thm} The inclusion maps for SLAG and PONT into the Grassmannian acts as follows.  
\begin{enumerate}
\item The inclusion map 
$i : SLAG \to G^+_3\mbbr^6$ of the 5-manifold of special Lagrangian planes act on the cohomology ring as follows.
$$i^*x_4=0,  ~~ i^*x_5=x_5,  ~~i^*y_3=x_3,  ~~i^*y_7=0.$$

\item The inclusion map 
$i : PONT \to G^+_3\mbbr^6$ of the 4-manifold of Pontrjagin cycles acts on the cohomology ring in a trivial way. 
\end{enumerate}
\end{thm}

\begin{proof} There is no cohomology at the levels 4 and 7. Since the evaluation of $x_5$ 
on SLAG is 1, it is mapped onto the generator at the top level.
$y_3$ is represented by the Euler class $eE(3,6)$ of the tautological bundle of $G_3^+\mbbr^6$.  

Considering the $PONT$, since there is no cohomology 
corresponding to the levels 3,5 and 7, the pullback is 
trivial on them. Finally at the 4-th level the pullback 
acts $i^*: \mbbz_2 \to \mbbz$ a homomorphism to the free space hence trivial s well. So that $i^*=0$.
\end{proof}


\section{Special Lagrangian Free Subsurfaces}\label{secslfreedim2}

In this section we are going to analyse the 
special Lagrangian-free 
submanifolds of $\mathbb{C}^2$. 
A closed orientable submanifold $M^k$ for $1\leq k \leq 2n-2$ of a 
Calabi-Yau manifold $(X^{2n}, Re(\Omega))$ is called {\em special 
Lagrangian-free} ({\em SL-free} in short) 
if there are no
special Lagrangian $n$-planes tangential to $M$. 
A result of 
\cite{HLpotentialtheory} states that
if $(M,\Omega)$ is a Calabi-Yau manifold of  real $2n$-dimension with special Lagrangian calibration $Re(\Omega)$, then the free dimension is  $fd(Re(\Omega))=2n-2.$

\vspace{.05in}

Our aim is to find the obstructions to embed any closed oriented manifold into the flat Calabi-Yau manifold $\mathbb{C}^n\cong \mathbb{R}^{2n}$ as SL-free. In order to achieve this, we will use the Gauss map of any embedding and intersection theory. \\
Now, let  $f: M^k\longrightarrow \mathbb{R}^{2n}$ be an embedding of a closed oriented 
k-dimensional manifold $M^k$ into $\mathbb{R}^{2n}$, where $0< k \leq 2n-2$. If $k< n$, then this embedding is automatically SL-free due to the fact that the dimension of a special Lagrangian plane is $n$. Locally, it will also be SL-free for $k=n$. For $n\leq k$, consider the Gauss map  ${\mathcal{G}_f}: M^{k}\longrightarrow G_k^+\mathbb{R}^{2n}$, and let $\mathcal{S} \in G_k^+\mathbb{R}^{2n}$ be the subset of k-planes which contains a special Lagrangian n-plane. If 
${\mathcal{G}_f}(M^{k})\cap \mathcal{S} =\emptyset $, then this embedding will be SL-free, too. If the intersection is non-empty, then we will try to find a topological invariant of $M^k$ which will be an obstruction for this embedding to be SL-free and we plan to do this using this intersection set. Unfortunately, this can easily be done, as far as we know, if the intersection of ${\mathcal{G}_f}(M^{k})$ and $\mathcal{S}$ is a set of points which only occurs if $dim(\mathcal{S})+dim({\mathcal{G}_f}(M^{k}))= dim(\mathcal{S})+k=dim(G_k^+\mathbb{R}^{2n}) $. Since ${\mathcal{G}_f}(M^{k})$ and $\mathcal{S}$ are closed, generically they will intersect at finitely many points under this condition. By using the weak Whitney embedding theorem and some classical results in differential topology, these intersections can be made transversal. Hence, we can compute the algebraic intersection numbers between them and then try to find conditions on $M^k$ which may make these numbers equal to zero. First, we will find the conditions on $k$ and $n$ when this intersection can generically be at just  points. This is clarified after the following result.
\begin{lem}
	Let $\mathcal{S} \in G_k^+\mathbb{R}^{2n}$ be the subset of k-planes which contain a special Lagrangian n-plane. Then, $dim(\mathcal{S})+k= dim(G_k^+\mathbb{R}^{2n}) $ if and only if $(k,n)=(2,2)$ or  $(k,n)=(6,5)$.
\end{lem}

\begin{proof} Let us define the following,  
	$$\mathcal{S} = \{V^k\in G_k^+\mathbb{R}^{2n} \lvert \ \  V^k \text{contains a special Lagrangian n-plane}  \}.$$
It is clear that $k\geq n$. Then, we can define the map $\pi : \mathcal{S} \rightarrow SLAG_n$ which maps each k-plane $V^k$ to the special Lagrangian plane contained. If we look at the fibers, i.e. $\pi^{-1}(Q)$ for any special Lagrangian n-plane $Q$ in $\mathbb{C}^n,$ we see that any k-plane $V^k \in \pi^{-1}(Q) $ will be of the form $V^k=Q\oplus \Lambda$ where $\Lambda$ is a (k-n) dimensional plane i.e. $\Lambda^{k-n} \in G_{k-n}^+\mathbb{R}^n$.
	\[
	\begin{tikzcd}[row sep=2.5em]
	G_{k-n}^+\mathbb{R}^n \arrow{r} & \mathcal{S}\arrow{d} \\
	&SLAG_n
	\end{tikzcd}
	\]
\ni Hence, this fibration gives us the following, 	
$$dim(\mathcal{S}) =dim (SLAG_n) + dim (G_{k-n}^+\mathbb{R}^n).$$
\ni Because of the quotient structure $SLAG_n\cong\mbb{SU}_2/\mbb{SO}_2$ we can compute the dimension as $\displaystyle dim(SLAG_n)=(n^2+n-2)/2$. Then the equation $dim(S)+k= dim(G_k^+\mathbb{R}^{2n}) $ can be written as
	$$ \frac{(n^2+n-2)}{2} + (k-n)(2n-k) + k = (2n-k)\cdot k $$ 
\ni which will turn to
	$$ -3n^2+(1+2k)n+2k-2=0$$
	$$ n_{1,2}=\frac{-(1+2k)\pm \sqrt{4k^2+28k-23}}{-6}$$ 
\ni where $n\leq k \leq 2n-2.$
The only integer solutions to this equation are $(k,n)=(1,0),(2,2),(6,5).$ Obviously $(k,n)=(1,0)$ is not a geometric object.  
\end{proof}

For the case $(k,n)=(6,5)$, we don't have enough tools to compute this intersection number as the dimensions are really big. However, in the low dimensional case $(k,n)=(2,2)$ we completely classified which closed orientable surfaces can be embedded into $\mathbb{C}^2\cong\mathbb{R}^4$ as SL-free.
\begin{thm}
	Let  $M$ be a closed orientable surface. If $M$ can be embedded into $\mathbb{R}^4\cong \mathbb{C}^2$ as a SL-free submanifold, then the Euler characteristic of $M$, $\chi (M) = 0.$
\end{thm}

\begin{proof} Let $f:M \rightarrow \mathbb{R}^4$ be an embedding 
of an oriented surface $M$ into $\mathbb{R}^4$, and
$G_f:M\rightarrow G_2^+\mathbb{R}^4$ be its corresponding Gauss map. Then, we have, 
$$ {G_f}_{*}[M]=\frac{1}{2} \chi (M)[G_2^+\mathbb{R}^3]= -\frac{1}{2}\chi (M)[\mathbb{CP}_1] -\frac{1}{2}\chi (M)[\overline{\mathbb{CP}}_1].$$ 
\ni $SLAG_2\cong \mbb{SU}_2/\mbb{SO}_2\cong S^2\cong \mathbb{CP}_1$ and $\widetilde{SLAG_2} \cong \overline{\mathbb{CP}}_1$ (the set of special Lagrangian 2-planes with opposite orientation)   are 2-dimensional submanifolds of the oriented Grassmannian of 2-planes in $\mathbb{R}^4$, i.e. $G_2^+\mathbb{R}^{4}$. For their embeddings into $G_2^+\mathbb{R}^{4}$, see 
\cite{shizhou}. Generically,  ${G_f}_{*}[M]$ will intersect with $SLAG_2$ and $\widetilde{SLAG_2}$ at finitely many points. 
If both of the intersections are empty, then $f$ is a SL-free embedding of $M$ into $\mathbb{R}^4$.  
Otherwise, we can count these intersections with sign or compute both of the algebraic intersection numbers between $ {G_f}_{*}[M]$ and $SLAG_2$ or $\widetilde{SLAG_2}$. 
These intersection numbers will give topological conditions on $M$ to make algebraic intersection numbers equal to zero. 
However, as it can be seen from the equation above, the 
intersection numbers between the classes $[SLAG_2]$, 
$[\widetilde{SLAG_2}]$ and $[{G_f}_{*}[M]]$ actually just depends on the Euler characteristic of $M$. 
The intersection number of $[SLAG_2]\cdot [{G_f}_{*}[M]]= a$ and 
the number $a$ must  algebraically be $0$ since it is SL-free. 
Therefore, $\frac{1}{2} \chi (M)[G_2^+\mathbb{R}^3]=0$ and hence 
$\chi (M)=0$. The idea is the same for $\widetilde{SLAG_2}$ as 
well. Hence, in order to have a SL-free embedding of $M$ into 
$\mathbb{R}^4$, $\chi (M)$ has to vanish. 
\end{proof}
Moreover,  we can actually see that the converse of the theorem 
is true, too. As a complementary part, we use the  idea of the 
Joyce in \cite{joyce2007riemannian}. The case  $n=2$ is 
actually the special case of special Lagrangian in 
$\mathbb{C}^n$. 
Let $\mathbb{C}^2$ have complex coordinates $(z_1,z_2),$ 
complex structure $J$, and metric $g$, K\"{a}hler form $\omega$, and holomorphic 2-form $\Omega$ defined in $\mathbb{C}^2$. Define real coordinates $(x_0,x_1,x_2,x_3)$ on $\mathbb{C}^2\simeq \mathbb{R}^4$ by $z_1=x_0 +ix_1$, $z_2=x_2 +ix_3$. Then, $$g= dx_0^2 +...+ dx_3^2,\ \omega= dx_0\wedge dx_1 + dx_2\wedge dx_3$$
	$$Re(\Omega)=dx_0\wedge dx_2-dx_1\wedge dx_3 \ ~~\text{and}~~ \ Im(\Omega)= dx_0 \wedge dx_3 +dx_1\wedge dx_2.$$
\ni	Now, define a different set of complex coordinates $(w_1,w_2)$ on $\mathbb{C}^2\simeq \mathbb{R}^4$ by  $w_1=x_0+ix_2$ and $w_2=x_1-ix_3.$ Then $\omega-iIm\Omega=dw_1\wedge dw_2$.
	But by proposition , a real 2-submanifold $L\subset $ is special Lagrangian if and only if $\Omega|_L=Im\Omega|_L=0.$ Thus, $L$ is special Lagrangian if and oly if $(dw_1\wedge dw_2)|_L=0.$ But this holds if and olny if $L$ is holomorphic curve with respect to the complex coordinates $(w_1,w_2).$ 

As a result, the second part of proof  shows that being SL-free with respect to the standard complex structure $J$ is equivalent to being totally real with respect to the complex structure $\bar{J}$ for $n=2$. Since, $h$-principle holds for totally real embeddings (see Gromov \cite{gromov1986partial}) and the only obstruction for SL-free embeddings is $\chi(M)$, we get the following result. 

\begin{cor}\label{G2R4SLfree} A closed orientable surface $M$ can be embedded into $\mathbb{R}^4\cong \mathbb{C}^2$ as a SL-free submanifold if and only if the Euler Characteristic of $M$, $\chi (M) = 0.$ 
\end{cor}

By using the classification of orientable surfaces, we see that only $\mathbb{T}^2= S^1\times S^1$ can be embedded into 
$\mathbb{R}^4\cong \mathbb{C}^2$ as SL-free. 

\newpage

\section{The Geometry}\label{secresim}

In this section, we analyze the geometry of $SLAG$ in $G_3^+\mbbr^6.$ For any vector $v\in G_3^+\mbbr^6,$ there are orthonormal vectors $e_1\wedge\cdots\wedge e_6$ and there exists $\phi\in \bigwedge^3(\mbbr^6).$ 
\begin{thm}(\cite{zhou2005morse}) 
The function   $\Phi : G_3^+\mbbr^6\rightarrow \mbbr$ defined by   $\Phi(v)=\langle v,\phi \rangle, \ \ \forall v\in G_3^+\mbbr^6$
is a degenerate Morse function for almost every form in $\bigwedge^3(\mbbr^6)$ where $\langle,\rangle$ is the inner product on $\bigwedge^3(\mbbr^6).$ The critical submanifolds are $\Phi^{-1}(1)$ and $\Phi^{-1}(-1)$ with indices $4$ and $0,$ respectively.
\end{thm}
Let $\phi$ be a closed $3$-form on Euclidean space $\mbbr^6.$ We call a calibration on $\mbbr^6$ \cite{harveylawson} if $\phi(v_1\wedge v_2\wedge v_3)\leq 1$ for any $v_i\in \mbbr^6.$ The set $\{v_1\wedge v_2\wedge v_3\in G_3^+\mbbr^6 \ |\ \ \phi(v_1\wedge v_2\wedge v_3)=1\}$ is called the face of calibration $\phi.$ Thus the face of $\phi$ is a critical submanifold of $\Phi: G_3^+\mbbr^6\rightarrow \mbbr$ defined by $\Phi$ at Figure \ref{figure:slag}.

\begin{figure} [ht]  \begin{center} 
\includegraphics[width=.8\textwidth]{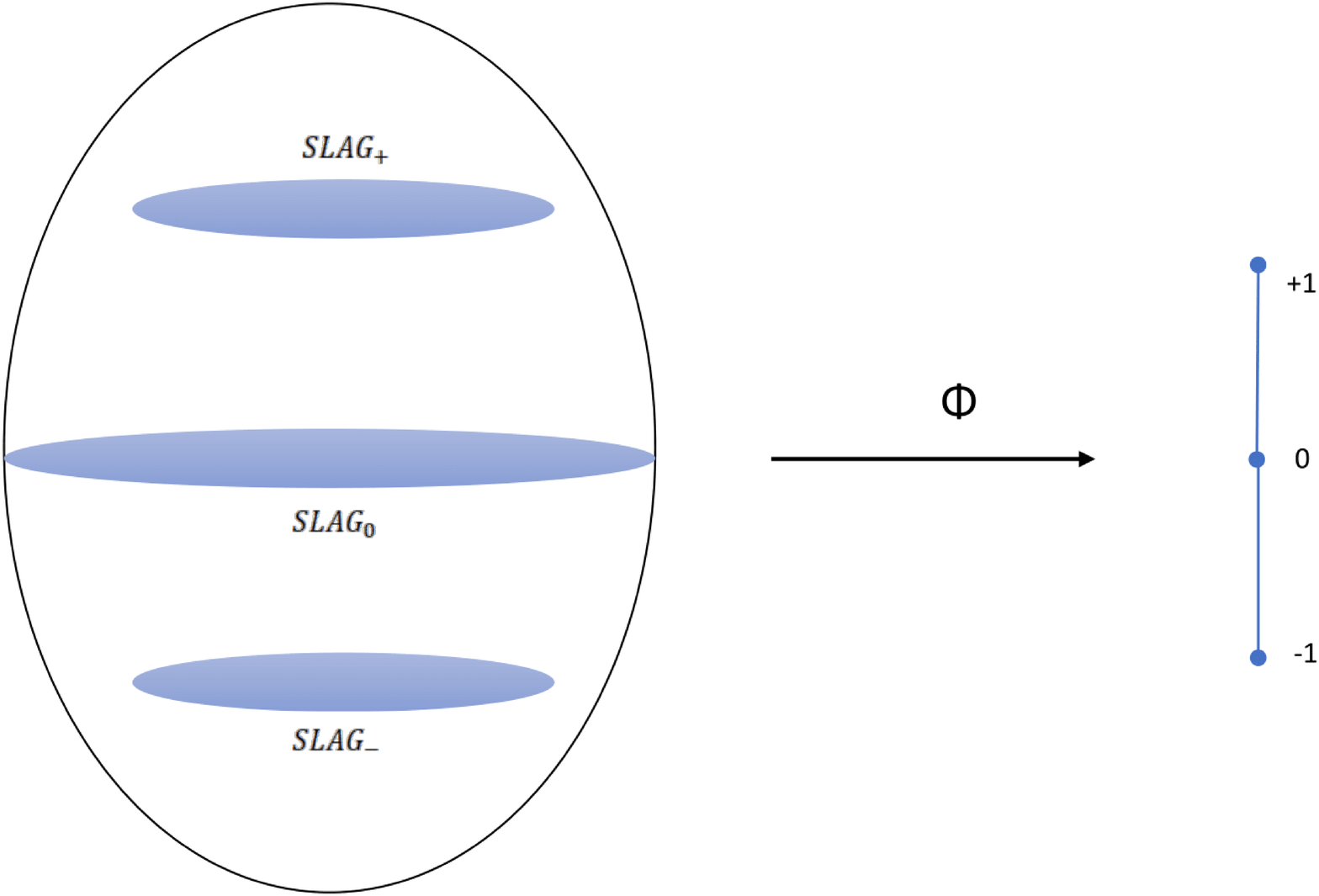}  \end{center} \caption{\, The map\, $\Phi: G^{+}_{3}\R^6 \to \R$ }
  \label{figure:slag}
\end{figure}
\ni Now, let $J$ be the complex structure on Euclidean space 
$\mbbr^6$ and $e_1,...,e_6$ be an orthonormal basis  with 
$Jdx^{2i-1}=dx^{2i}.$ We define the calibration  $3$-form by the real part of the holomorphic volume form as follows. 
\begin{eqnarray*}
\phi & = & Re[(dx^1+idx^2)\wedge \cdots \wedge (dx^5+idx^6)]\\ [2\jot]
 & = & dx^{135}-dx^{146}-dx^{236}-dx^{245}. \nonumber
\end{eqnarray*}
\ni $\phi$ is called a special Lagrangian calibration on $\mbbr^6$ with the face of set ${SLAG}=\mbb{SU}_3/\mbb{SO}_3.$ Critical submanifolds are the following; $$ SLAG_+=\{v \in  G^{+}_{3}\R^6 \ \lvert\ \  \phi|_ v=vol(v)\},$$  $$SLAG_-=\{v \in  G^{+}_{3}\R^6 \ \lvert\ \  \phi|_ v=-vol(v)\}, $$  
$$ SLAG_0=\{v \in  G^{+}_{3}\R^6 \ \lvert\ \  \phi|_ v=0\}.$$
The equator which is the zero level hypersurface $SLAG_0$ is an 8-dimensional submanifold with vanishing homology class. Actually, the critical submanifold $SLAG_0$ is isomorphic to the Lie group $SU(3)$ since $SU(3)$ acts on $v \in SLAG_0,$ (i.e. the complex structure is preserved) and there is no other restriction holds.




\section{Application to the normal bundles}\label{secnormalbundle}
In this section we make an application to embeddings. 
Let $i:M \longrightarrow \mbb R^6$ be an immersion of a 
3-manifold into the 
Euclidean space. We have the following theorem. 

\begin{thm}\label{G3R6SLfree} 
Let $M^3$ be a closed, oriented 3-manifold and 
$i:M\to\mbbr^6$ be an immersion, then the image $g_N(M)$ of 
the normal bundle under the normal Gauss map $g_N :M\to G_3^+\mbbr^6$ is
contractible and the normal bundle of the immersed 
submanifold is trivial, the  immersion is generically special 
Lagrangian free .  \end{thm}
\begin{proof}
We are going to use obstruction theory. 
We shrink the normal Gauss map $g_N :M\to G_3^+\mbbr^6$ skeleton by skeleton. The restriction of $g_N$ to the 0-th and 1-st skeleton of $M$ can be contracted to a point by a homotopy because the image lies in the Grassmannian $G_3^+\mbbr^6$ which is a connected and simply connected space. After this homotopy we obtain a map, 
$$g_N : M_{(2)}/M_{(1)}\longrightarrow G_3^+\mbbr^6.$$
\ni Since this map is shrinked over the 1-skeleton, it defines a 2-cochain which is closed by obstruction theory and hence 
defines a class in the second cohomology of $M$ with $\pi_2$ 
coefficients,which is computed to be $\mbbz_2$ in the previous
section. 
$$\mf o_2\in H^2(M;\{ \pi_2\, G_3^+\mbbr^7\})=H^2(M;\mbbz_2).$$
\ni The second Stiefel-Whitney class $\omega_2$ of the 
3-manifold is equal to this obstruction. Since oriented 
3-manifolds are parallelizable, all the characteristic classes
vanish, in particular the Stiefel-Whitney classes.  The next 
obstruction is, 
$$\mf o_3\in H^3(M;\{ \pi_3\, G_3^+\mbbr^7\})=H^3(M; \mbbz_2 ),$$
\ni by the Lemma \ref{lempi3}.
Since the 3rd Stiefel-Whitney class is zero this obstruction also vanishes and 
the normal Gauss map is contractible. 

Since the free dimension for the special Lagrangian calibration is 2n-2=4 in this case, a 3-manifold is generically an SL-free submanifold by \cite{HLpotentialtheory}. \end{proof}


\bigskip

\bibliographystyle{alphaurl}
\bibliography{slag}

\bigskip

{\small
\begin{flushleft}
\textsc{Orta mh. Z\"{u}beyde Han{\i}m cd. No 5-3 Merkez 74100 Bart\i n, T\" urk\'{i}ye.}\\
\textit{E-mail address:} \texttt{\textbf{kalafat@\,math.msu.edu}}
\end{flushleft}
}

{\small 
\begin{flushleft} \textsc{Orta Do\u gu Tekn\' \i k  \" Un\' \i vers\' ites\' i, 06800, Ankara, T\" urk\'{i}ye.}\\
\textit{E-mail address:}  \texttt{\textbf{e142649@\,metu.edu.tr}}
\end{flushleft}
}

\end{document}